\documentclass[11pt]{amsart}
\usepackage[margin=3cm]{geometry}               
\usepackage[usenames,dvipsnames,svgnames,table]{xcolor}    
\usepackage{lineno, hyperref}
\usepackage{amssymb,amsfonts, amsmath, amsthm}
\usepackage{epstopdf,enumerate}
\usepackage{ulem}
\usepackage{cancel}
\usepackage{thmtools}
\usepackage{mathrsfs}
\usepackage{soul}
\usepackage[english]{babel}

\newtheorem*{theorem*}{Theorem}
\newtheorem*{proposition*}{Proposition}
\newtheorem*{corollary*}{Corollary}
\theoremstyle{definition}
\newtheorem*{definition*}{Definition}

\declaretheorem[style=definition,qed=$\dashv$,numberwithin=section]{definition}
\declaretheorem[style=definition,qed=,sibling=definition]{remark}
\declaretheorem[style=definition,qed=$\dashv$, sibling = definition]{lemma-definition}

\declaretheorem[style=theorem, sibling = definition]{theorem}
\declaretheorem[style=theorem, sibling = definition]{lemma}
\declaretheorem[style=theorem, sibling = definition]{proposition}

\declaretheorem[style=theorem, sibling = definition]{corollary}

\declaretheorem[style=definition, sibling = definition]{example}

\newcommand{\cal}{\mathcal}

\newcommand{\N}{\mathbb N}

\newcommand{\F}{\mathbb F}
\newcommand{\Fp}{\F_p}
\newcommand{\OO}{\mathcal{O}}
\newcommand{\res}{\mathrm{res}}
\newcommand{\Res}{\mathrm{Res}}
\newcommand{\Span}{\textup{Span}}

\newcommand{\sn}{\par\smallskip\noindent}

\title{On valuation independence and defectless extensions of valued fields}

\author[A. Blaszczok ]{Anna Blaszczok}
\address{\hskip-\parindent
Anna Blaszczok\\
Institute of Mathematics\\
University of Silesia\\ 
Bankowa 14, 40-007 Katowice, Poland.}
\email {anna.blaszczok@us.edu.pl
}

\author[P. Cubides Kovascics]{Pablo Cubides Kovacsics}
\address{\hskip-\parindent
Pablo Cubides Kovacsics \\ Laboratoire de math\'ematiques Nicolas Oresme\\ Universit\'e de Caen\\CNRS UMR 6139 
Universit\'e de Caen BP 5186\\
14032 Caen cedex, France. }
\email{pablo.cubides@unicaen.fr}

\author[F.-V. Kuhlmann]{Franz-Viktor Kuhlmann
}
\address{\hskip-\parindent
Franz-Viktor Kuhlmann\\
Institute of Mathematics\\ 
University of Szczecin
\\
ul. Wielkopolska 15, 70-451 Szczecin, Poland.\\
}
\email{fvk@usz.edu.pl} 
\begin{document}

\begin{abstract}
In this article we further develop the theory of valuation independence and study its relation with classical notions in valuation theory such as immediate and defectless extensions. We use this general theory to settle two open questions regarding vector space defectless extensions of valued fields. Additionally, we provide a characterization of such extensions within various classes of valued fields, extending results of Fran\c{c}oise Delon.  
\end{abstract}

\maketitle


\newcounter{eqn}
\normalem 

Valuation independence is a natural relation which strengthens linear independence in the framework of valued fields and valued vector spaces. Its definition appears in many different contexts of valuation theory and can be traced back to work of Robert \cite{robert1967}, which was based on work by Cohen and Monna \cite{cohen1948, monna1956, monna1957}. 

In this article we further develop the theory of valuation independence for general valued fields in the sense of Krull (\textit{i.e.}, of arbitrary rank) and study its relation with various classical notions in valuation theory such as immediate extensions and defectless extensions, among others. We use this general theory to settle two open questions regarding \emph{vector space defectless} (hereafter $vs$-defectless) extensions of valued fields, a type of extension introduced --in its most general form-- by Baur \cite{B} (under the name ``separated extension'') and further studied by Delon in \cite{D}. 

In the following section we present the main concepts and results of the article. 

\section{Main results}

Let $(K,v)$ be a valued field. We use the notation $vK$ for the value group, $\OO_K$ for the valuation ring, $Kv$ for the residue field and $\res$ for the residue map. By $(L|K,v)$ we denote an extension of valued fields: $L|K$ is a field extension, $v$ is a valuation on $L$ and $K$ is equipped with the restriction of $v$ to $K$. Every such extension induces canonical embeddings of $vK$ into $vL$ and of $Kv$ into $Lv$. Recall that if the canonical embeddings are onto, then the extension $(L|K,v)$ is called \emph{immediate}. In other words, $(L|K,v)$ is an immediate extension if the corresponding value group and residue field extensions are trivial. 

Throughout we will work over a valued field extension $(L|K,v)$ unless otherwise stated. Let $W\subseteq V$ be $K$-vector spaces with $V\subseteq L$. The valuation and the residue map induce respectively a totally ordered set $vV\subseteq vL$ and a $Kv$-vector subspace $Vv:=\res(\OO_V)$ of $Lv$ where $\OO_V:=\{a\in V\mid v(a)\geq 0\}$. We say that the $K$-vector space extension $W\subseteq V$ is \emph{finite} if $\dim_K V/W$ is finite. 

\begin{definition}\label{def:val_indep} A subset $B\subseteq V\setminus\{0\}$  is \emph{$(K,v)$-valuation independent over $W$} if for every finite $K$-linear combination $\sum_{i=1}^n c_ib_i$ of (pairwise distinct) elements $b_i\in B$ and every $a\in W$, we have that 
\[
v\left(\sum_{i=1}^n c_ib_i+a\right)=\min_{1\leqslant i\leqslant n}\{v(c_ib_i),v(a)\}.
\]
\end{definition} 

\begin{remark}\label{rmk:valInd} Note that if $B$ is $(K,v)$-valuation independent over $W$ then it is $K$-linearly independent over $W$. Indeed, for any $a\in W$ and a finite $K$-linear combination $b:=\Sigma_{i=1}^n c_ib_i$ with $b_i\in B$  and $c_i\in K$ such that $b+a=0$ we have that  
\[
\infty=v(0)=v\left(\sum_{i=1}^n c_ib_i+a\right)=\min_{1\leqslant i\leqslant n} \{v(c_ib_i), v(a)\},  
\]
which imposes that $a=0$ and all $c_i=0$. That the converse does not hold will be later shown as a special case of Lemma~\ref{lem:immediate_valdep}.  
\end{remark}

As usual, given a $(K,v)$-valuation independent set $B\subseteq V$, if $V=\Span_K(B)\oplus W$ we say that $B$ is a \emph{$(K,v)$-valuation basis of $V$ over $W$}. The set $B$ is $(K,v)$-valuation independent (resp. a $(K,v)$-valuation basis of $V$) if it is $(K,v)$-valuation independent over $W=\{0\}$ (resp. a $(K,v)$-valuation basis of $V$ over $\{0\}$). It is called \emph{$(K,v)$-valuation dependent over $W$} if it is not $(K,v)$-valuation independent over $W$. When the valued field $(K,v)$ in consideration is clear from the context, we will often omit $(K,v)$ and simply say $K$-valuation independent, $K$-valuation basis, etc. 

In Section \ref{sec:valind}, the general theory of the notion of valuation independence is developed. Some of the results hold in a slightly more general context which is presented in the Appendix. A dagger sign $(\dagger)$ will be added in front of those results which hold in this broader setting. In those cases, the same proofs work with minor modifications.   

Various results in this section can be found in the literature but, more often than not, whithin a less general setting (for example, they are proved only for rank 1 valued fields as in \cite{BGR}, or only for valued vector spaces where the scalar field $K$ is trivially valued as in \cite{fuchs1975, KS}). The benefit of gathering these results here is twofold. On the one hand, the common general framework we provide unifies results and terminology which radically change from author to author, making it easier to establish the subject's state of the art. On the other hand, all our proofs (in Sections \ref{sec:valind} and \ref{sec:defectext}) rely only on algebraic methods and basic knowledge of valuation theory, reducing the background needed to prove them. Some of our contributions in Section \ref{sec:valind} include the introduction of the notion of \emph{normalized valuation independent set} (see Subsection~\ref{subsec:normal}) and the following characterization of immediate extensions: 
  
\begin{proposition*}[Later Proposition \ref{pro:char-immediate}]
Take a valued field extension $(M|K,v)$. Then the following are equivalent:
\begin{enumerate}
\item $(M,v)$ is an immediate extension, 
\item for every subset $\{b_1,\ldots,b_n\}$ of some valued field extension of $(M,v)$, if $b_1,\ldots,b_n$ are $K$-valuation independent, then they are also $M$-valuation independent. 
\end{enumerate}
\end{proposition*}

Let us now recall the definition of $vs$-defectless extensions. 

\begin{definition}\label{def:vsdefect} The extension $(L|K,v)$ is called $vs$-defectless (vector-space defectless) if every finitely generated $K$-vector subspace of $L$ has a $K$-valuation basis. 
\end{definition}

\begin{remark} The previous definition is due to Baur \cite{B}, who originally called such extensions \emph{separated} extensions. A similar notion for rank 1 valued fields appears in \cite{BGR} under the name of \emph{weakly stable fields}. Unfortunately, both choices of terminology conflict with standard vocabulary from other areas of mathematics which have a strong connection to valuation theory (in particular, algebraic geometry and model theory). The term `$vs$-defectless' chosen in this article was coined during the eighties by Roquette's group in Heidelberg. Green, Matignon and Pop in \cite{GMP} introduced a vector space defect for a special sort of valued function fields which is trivial if and only if the  function field is a vs-defectless extension.
\end{remark}

Section \ref{sec:defectext} is devoted to the study of defectless and $vs$-defectless extensions of valued fields using the tools introduced in Section \ref{sec:valind}. In the first part of Section \ref{sec:defectext}, we provide the following characterization of defectless extensions (all terms to be later defined). 

\begin{proposition*}[Later Proposition \ref{char_defectless}]
Assume that the extension $(L|K,v)$ is finite. Then the following conditions are equivalent.
\begin{enumerate}
\item $[L:K]=(vL:vK)[Lv:Kv]$,
\item $(L|K,v)$ admits a standard $K$-valuation basis,
\item $(L|K,v)$ admits a $K$-valuation basis, 
\item $(L|K,v)$ is a $vs$-defectless extension.
\end{enumerate}
\end{proposition*}

The second part of Section \ref{sec:defectext} deals with arbitrary $vs$-defectless extensions (not necessarily finite). In particular, we study the logical implications of the following properties of a valued field extension $(L|K,v)$:

\begin{enumerate}
\item[(A)] the extension $(L|K,v)$ is $vs$-defectless;
\item[(B)] for every $K$-vector space $V\subseteq L$ of finite dimension and every $a\in L$, the set~$\{v(a-x)\mid x\in V\}$ has a maximal element; 
\item[(C)] $L$ is linearly disjoint over $K$ from every immediate extension $M$ of $K$ (in every common field extension over $K$). 
\end{enumerate}

In \cite{D}, Delon proved the following theorem. 

\begin{theorem*}[Delon] For any valued field extension  $(L|K,v)$, $(B)\Rightarrow(A)\Rightarrow (C)$.  
\end{theorem*}

Delon's proof of $(A)\Rightarrow (C)$ uses tools from the model theory of pairs of valued fields as studied by Baur in \cite{B}. It remained open whether implications $(A)\Rightarrow (B)$ and $(C)\Rightarrow (A)$ hold in general. We answer both questions by showing that the former implication does hold in general, while the latter does not. An example of a valued field extension that does not satisfy the implication $(C)\Rightarrow (A)$ is given in Proposition \ref{prop:notCA}. The general theory developed in Section \ref{sec:valind} allows us to provide a fully algebraic proof of the following: 

\begin{theorem*}[Later Theorem \ref{thm:main}] Let $(L|K,v)$ be an extension of valued fields. Then $(A)\Leftrightarrow (B) \Rightarrow (C)$. 
\end{theorem*}

In Section \ref{sec:CA} we study various instances where the implication $(C)\Rightarrow (A)$ does hold. A first example was already given by Delon in \cite{D}, where she showed that if $(K,v)$ is an algebraically maximal Kaplansky field, then $(C)\Rightarrow (A)$ for any valued field extension $(L|K,v)$. Unfortunately, a gap was found in her proof. However, we recover her theorem as a special case of the following abstract criterion (thus preventing a snowball effect of incorrect proofs, as her result was used by the second author and Delon in \cite{cubi-delon}). 

\begin{theorem*}[Later Theorem \ref{thm:criterion}]
Suppose $\cal K$ is an elementary class of valued fields having the following properties:
\begin{enumerate}
\item[(P1)] every member of $\cal K$ is existentially closed in each of its maximal immediate extensions,
\item[(P2)] all maximal immediate extensions of members of $\cal K$ are again members of $\cal K$,
\item[(P3)] if $(K,v)\in \cal K$ and $(F,v)$ is a relatively algebraically closed subfield such that $(K|F,v)$ is
immediate, then $(F,v)\in \cal K$.
\end{enumerate}
Then, for all $(K,v)\in \cal K$, every extension $(L|K,v)$ satisfies the implication $(C)\Rightarrow (A)$. 
\end{theorem*}

Classes $\cal K$ satisfying the conditions of the previous theorem include:

\begin{enumerate}
\item[$\bullet$] the class of all tame valued fields (which includes the class of all algebraically maximal Kaplansky fields),
\item[$\bullet$] the class of all henselian finitely ramified fields (which includes the class of all $\wp$-adically closed fields).
\end{enumerate}

To conclude, we show the following result which in particular covers the situation of rank 1 discretely valued fields. 

\begin{theorem*}[Later Theorem \ref{thm:cofinal}] Let $(L|K,v)$ be such that: 
\begin{enumerate}
\item $\widehat{K}$ (the completion of $K$) is the maximal immediate extension of $K$ and 
\item $vK$ (the value group of $K$) is cofinal in $vL$. 
\end{enumerate}
Then $(C)\Rightarrow (A)$. 
\end{theorem*}

We would like to acknowledge that very recently, Romain Rioux obtained independently a proof of implication $(A)\Rightarrow (B)$ for arbitrary valued field extensions as a byproduct of results in his PhD thesis \cite{rioux}. Although the result does not appear in \cite{rioux}, he communicated to us that the key propositions from which it can be derived are {\cite[Propositions 2.3.13 and 2.3.16]{rioux}}. His approach is however different from ours. 


\section{Valuation independence in valued vector spaces}\label{sec:valind}

We work over a valued field extension $(L|K,v)$ and we let $W\subseteq V$ be $K$-vector spaces with $V\subseteq L$. The following notation will be used throughout the paper. For subsets $X,Y $ of $L$, we set $v(X)+v(Y):=\{v(x)+v(y)\mid x\in X, y\in Y\}$. For a subset $U\subseteq V$ and $a\in L$, we define
\begin{eqnarray*}
v(U) & := & \{v(b)\, |\, b\in U\},\\
\res(U,a) & := & \{ \res(a'/a) \, | \, a'\in U \textrm{ and } v(a')=v(a)\}.
\end{eqnarray*}
We let $\Res(U,a)$ denote the \emph{multiset} $\{ \res(a'/a)\ | \ a'\in U \textrm{ and } v(a')=v(a)\}$, that is, we allow repetition of elements. This distinction between $\res(U,a)$ and $\Res(U,a)$ will be particularly useful concerning linear independence, as it may well be the case that $\res(U,a)$ is a $K$-linearly independent set while $\Res(U,a)$ is not (for instance when $\res(U,a)$ contains a unique element which is repeated in $\Res(U,a)$). Note that the identity
\[
\Span_{Kv}(\res(U,a))=\Span_{Kv}(\Res(U,a)),
\]
always holds. We let $vW:=v(W)\setminus\{\infty\}$. 

\subsection{Basic properties}

The next lemma follows immediately from the definition of a valuation basis.
\begin{lemma}	\label{val_basis_propert}
For every $K$-valuation basis $B:=\{b_i\, |\, i\in I\}$ of $V$ and $a\in L$ we have
\begin{enumerate}
\item \emph{($\dagger$)} $v(V)= v(K)+v(B)$,
\item $\res(V,a)= \Span_{Kv}(\res(B,a))$,
\item\emph{($\dagger$)} for every $\{c_i\, | \, i\in I\}\subset  K^{\times}$ the system $\{c_ib_i\, | \, i\in I\}$ is a $K$-valuation basis of $V$. 
\end{enumerate}
\end{lemma}

An important example of a valuation independent set is given by the following result. 

\begin{lemma}[{\cite[Lemma~3.2.2 ]{prestel2005}}]\label{algvind_1}
Let $X\subseteq L$ be such that for any two elements in $X$, their image under the valuation belong to
distinct cosets modulo $vK$. Let $Y\subseteq \OO_L$ be such that $\Res(Y,1)$ is $Kv$-linearly independent. Then the set $B:=\{xy\mid x\in X, y\in Y\}$ is $K$-valuation independent. 
\end{lemma}

If in addition  $1\in X$ and $1\in Y$, then the set $B$ will be called a \emph{standard $K$-valuation independent set}. Compare the previous situation with the following lemma.

\begin{lemma} \label{lem:immediate_valdep}
Assume that $(L|K,v)$ is an immediate extension of valued fields. Then every two elements $a,b\in L^\times$ are $K$-valuation dependent.
\end{lemma}
\begin{proof}
Take $a,b\in L^\times$. Since by assumption $vL=vK$, there is $c\in K$ such that $v(ca)=v(b)$. Hence $v(\frac{ca}{b})=0$. As $Lv=Kv$, there is an element $x\in K$ such that $\res(\frac{ca}{b})=\res(x)$. Then $v(\frac{ca}{b}-x)>0$ and consequently
\[
v(ca-xb)>vb=\min\{v(ca),v(xb)\}.
\]
This shows that $a,b$ are $K$-valuation dependent. 
\end{proof}

\begin{definition}[$\dagger$]\label{def:vvs_immediate} An extension $W\subseteq V$ is called \emph{immediate} if for all $a\in V\setminus\{0\}$ there is $b\in W$ such that $v(a-b)>v(a)$. 
\end{definition}

For $a\in V$ we set 
\[
v(a-W):=\{v(a-b) \mid b\in W\}. 
\]

\begin{lemma}[$\dagger$]\label{lem:vvs_immediate}
An extension $W\subseteq V$ is immediate if and only if for all $a\in V\setminus W$ the set $v(a-W)$ has no maximal element. 
\end{lemma}
\begin{proof}
Assume that the extension $W\subseteq V$ is immediate. Fix $a\in V\setminus W$ and take $c\in W$. Then $a-c\in V$, so by definition of immediate extensions, there is $b\in W$ such that $v(a-c-b)>v(a-c)$. Since $c+b\in W$, this shows that $v(a-W)$ has no maximal element.

For the converse, take $a\in V\setminus \{0\}$.  We wish to find $b\in W$ such that $v(a-b)>v(a)$. If $a\in W$, then obviously $v(a-a)>v(a)$. Otherwise, by assumption $v(a-W)$ has no maximal element, so in particular there is $b\in W$ such that $v(a-b)>v(a-0)=v(a)$.
\end{proof}

\begin{remark}
Note that for all $a\in L$, if $v(a-W)$ has no maximal element, then $v(a-W)\subseteq vW$. Indeed, for any element $b\in W$ we can find $c\in W$ such that $v(a-b)<v(a-c)$. Thus
$v(a-b)=v(a-b-(a-c))=v(c-b)\in vW$. 
\end{remark}

\begin{lemma} 
If the extension $W\subseteq V$ is immediate, then 
\begin{enumerate}
\item \emph{($\dagger$)} $vW=vV$;
\item  $Wv=Vv$.
\end{enumerate}
\end{lemma}
\begin{proof}
Assume that $W\subseteq V$ is an immediate extension. Take an element $a\in V\setminus W$. Then by definition 
the set  $v(a-W)$ admits no maximal element. By the previous remark, this yields that $v(a-W)\subset vW$. Hence $v(a)=v(a-0)\in vW$. If moreover $v(a)=0$, then since $v(a-W)$ admits no maximal element, $v(a-b)>0$ for some $b\in W$. Thus $\res(a)=\res(b)\in Wv$. 
\end{proof}

The converse of the previous lemma is not true as shows the following example: 

\begin{example} Let $K=\Fp (t)$ with the $t$-adic valuation. Let $y$ be transcendental over $K$. Then there is a unique extension of $v$ to $K(y)$ with $v(y)=0$ and $res(y)$ transcendental over $Kv$, namely 
\begin{equation*} 
v\left(\sum_{i=0}^n a_iy^i\right) =\min_{0\leq i\leq n}\{v(a_i)\}.
\end{equation*}
Set $W:=\Fp[t]$ and $V:=W\oplus Span_K(ty)$. The reader can easily show that $\N=vW=vV$ and that $\Fp= Wv=Vv$. On the other hand, for every element $a=\sum_{i=0}^n c_it^i\in W$ we have that  
\[
v(ty-a)=v\left(ty-\sum_{i=0}^n c_it^i\right)= \min\left\{v(t),\, v\left(\sum_{i=0}^n c_it^i\right)\right\}\leqslant 1=v(ty),
\]
which shows that $W\subseteq V$ is not immediate. 
\end{example}

In the case where $W=K$ and $V=L$, the above lemma together with \cite[Theorem 1]{kaplansky} of Kaplansky, gives the following classical characterization of immediate extensions. 

\begin{theorem} \label{charact_immediate} 
Let $(L|K,v)$ be a valued field extension. Then the extension is immediate if and only if for every $a\in L\setminus K$ the set $v(a-K)$ has no maximal element. 
\end{theorem}

\begin{lemma}[$\dagger$]\label{lem:vbasis_max} A set $B\subseteq V$ is $K$-valuation independent over $W$ if and only if for every $b=\sum_{i=1}^n c_ib_i+a$ with $c_i\in K$, $b_i\in B$ and $a\in W$, we have that 
\begin{equation}\label{eq:1}\stepcounter{eqn}\tag{E\arabic{eqn}}
\max v(b-W)=\min_{1\leqslant i\leqslant n} \{v(c_ib_i)\}.
\end{equation}
Under the previous assumptions, if \eqref{eq:1} holds, then $v(b-a)=\max(v(b-W))$. 
\end{lemma}

\begin{proof} Suppose first that $B$ is $K$-valuation independent over $W$. Set $I:=\{1,\ldots,n\}$. Then 
\[
\max(v(b-W))=\max_{x\in W}v\left(\sum_{i\in I} c_ib_i + x\right) = \max_{x\in W} \min_{i\in I}\{v(c_ib_i), v(x)\} = \min_{i\in I}\{v(c_ib_i)\}.  
\]
For the converse, suppose that equation \eqref{eq:1} holds for all $b\in \Span_K(B)\oplus W$. In this case, for  $b=\sum_{i\in I} c_ib_i+a$ we have that 
\[
v(b)=v\left(\sum_{i\in I} c_ib_i +a\right) \leqslant \max(v(b-W)) = \min_{i\in I} \{v(c_ib_i)\}. 
\]
The ultrametric inequality implies that $v(b)=\min_{i\in I}\{v(c_ib_i),v(a)\}$. The last assertion follows directly using $\eqref{eq:1}$ and the assumption that $B$ is $K$-valuation independent. 
\end{proof}

\begin{corollary}[$\dagger$]\label{cor:Kvalbasis} If $V$ admits a $K$-valuation basis over $W$ then for every $a\in V\setminus W$ the set $v(a-W)$ has a maximal element. 
\end{corollary} 

\begin{lemma}[$\dagger$]\label{lem:basis_transitivity} 
Assume that $B,B'\subset V$. Then $B\cup B'$ is $K$-valuation independent over $W$ if and only if $B$ is $K$-valuation independent over $W$ and $B'$ is $K$-valuation independent over $W+\Span_K(B)$. 
\end{lemma}
\begin{proof}
Straightforward. 
\end{proof}

\begin{lemma}[$\dagger$] \label{val_basis_dim1}
Assume that $V$ admits a valuation basis $B$ over $W$. Then for any element $x\in V\setminus W$ there are $b\in B$ and $a\in W$ such that $\{x-a\}$ is $K$-valuation independent over $W$ and $B\setminus \{b\}$ is a $K$-valuation basis of $V$ over $W\oplus \Span_K(x)$.
\end{lemma}
\begin{proof}
Set $B=\{b_i\, |\, i\in I\}$. Take an element $x\in V\setminus W$. Then $x=a+\displaystyle\sum_{i\in I}c_ib_i$ for some $a\in W$ and $c_i\in K$ all but finitely many equal to zero. 
Set $y=x-a =\displaystyle\sum_{i\in I}c_ib_i$. Since $x\notin W$, we have that $y\neq 0$. As $B$ is $K$-valuation independent over $W$, for every $w\in W$  and $c\in K$ we have that 
\begin{equation}\label{partial_sum} \stepcounter{eqn}\tag{E\arabic{eqn}}
v(cy+w)=v\left(\sum_{i\in I}cc_ib_i +w\right)=\min_{i\in I}\{v(cc_ib_i),v(w)\}=\min \{v(cy),v(w)\}.
\end{equation} 
Hence $\{y\}$ is $K$-valuation independent over $W$. 

Choose an index $j\in I$ such that $c_jb_j$ is of minimal value among the summands $c_ib_i$, $i\in I$. We show that then $B'=B\setminus \{b_j\}$ is a $K$-valuation basis of $V$ over $W\oplus \Span_K(x)$. Since $c_j\neq 0$ we have that 
\[
W+ \Span_K(x)+\Span_K(B')= W+\Span_K(B)=V.
\]
Hence it is enough to show that $B'$ is $K$-valuation independent over the vector subspace $W\oplus \Span_K(x)=W\oplus \Span_K(y)$ of $V$. Take $c'_i\in K$, $i\in I$, all but finitely many equal to zero and $w+cy\in W\oplus \Span_K(y)$ with $w\in W$ and $c\in K$. 
Assume first that 
\[
v(cy)\neq v\left( \displaystyle \sum_{i\in I\setminus\{j\}} c'_ib_i+w\right).
\]
Then equation~\eqref{partial_sum} together with the $K$-valuation independence of $B$ over $W$ yields that
\begin{align*}
v\left( \sum_{i\in I\setminus\{j\}} c'_ib_i+(cy+w)\right) 	& = \min\left\{ v(cy), v\left( \sum_{i\in I\setminus\{j\}} c'_ib_i+w\right)\right\} \\  
															& = \min_{i\in I\setminus\{j\}}\{v(c'_ib_i), v(w), v(cy)\}= \min_{i\in I\setminus\{j\}}\{v(c'_ib_i), v(cy+w)\}.
\end{align*}
Assume now that 
\[
v(cy)=  v\left( \displaystyle \sum_{i\in I\setminus\{j\}} c'_ib_i+w\right) = \min_{i\in I\setminus\{j\}}\{v(c'_ib_i), v(w)\}. 
\]
We wish to show that 
\[
v\left( \sum_{i\in I\setminus\{j\}} c'_ib_i+(cy+w)\right)= \min_{i\in I\setminus\{j\}}\{v(c'_ib_i), v(cy+w)\} =v(cy),
\]
where the last equality follows from our assumption together with equation~\eqref{partial_sum}. Since $v(c_jb_j)=\displaystyle\min_{i\in I} v(c_ib_i)=v\left(\displaystyle\sum_{i\in I}c_ib_i\right)$, we have that $v(cy)=v(cc_jb_j)$. Thus from the $K$-valuation independence of $B$ over $W$ we obtain that
\begin{align}\label{sum_case2} \stepcounter{eqn}\tag{E\arabic{eqn}}	
v\left( \sum_{i\in I\setminus\{j\}} c'_ib_i+(cy+w)\right) & = & v\left( \sum_{i\in I\setminus\{j\}} (c'_i+cc_i)b_i+ cc_jb_j+w\right)\\  
																							& = & \min_{i\in I\setminus\{j\}}\{v((c'_i+cc_i)b_i), v(cc_jb_j), v(w)\}. \nonumber
\end{align} 

Note that by assumption $v(w)\geq v(cy)=v(cc_jb_j)$ and $v(c'_ib_i)\geq v(cy)$ for every $i\in I\setminus\{j\}$. Thus also $v((c'_i+cc_i)b_i)\geq \min\{v(c'_ib_i),v(cc_ib_i)\}\geq \min\{v(cy), v(cc_jb_j)\}=v(cy)$ for all $i\in I\setminus\{j\}$. Together with equation~\eqref{sum_case2} this yields that 
\[
v\left( \sum_{i\in I\setminus\{j\}} c'_ib_i+(cy+w)\right)= v(cy). 
\]
Hence we obtain that  $B'$ is $K$-valuation independent over $W\oplus \Span_K(x)$ and thus is a $K$-valuation basis of $V$ over $W\oplus \Span_K(x)$.
\end{proof}

\begin{lemma}[$\dagger$]\label{lem:val_bas_subext}
Assume that $V$ admits a $K$-valuation basis $B$ over $W$. If $W\subset W'$ is a finite subextension of $W\subset V$, then $W'$ admits a $K$-valuation basis $A$ over $W$. Moreover, there is $B'\subset B$ which is a valuation basis of $V$ over $W'$.
\end{lemma}
\begin{proof}
Take any $x\in W'\setminus W$. Then, by Lemma~\ref{val_basis_dim1} there is $b_1\in B$ and $a_1\in W$ such that $x_1:=x-a_1$ forms a valuation basis of $W_1:=W\oplus \Span_K(x)=W\oplus \Span_K(x_1)\subseteq W'$ over $W$ and $B_1:=B\setminus\{b_1\}$ is a valuation basis of $V$ over $W_1$.

Take $s\leqslant $dim$_K W'/W$. Suppose we have chosen $x_1,\ldots,x_s\in W'$ which are $K$-valuation independent over $W$ and $b_1,\ldots,b_s\in B$ such that $B_s:=B\setminus\{ b_1,\ldots,b_s\}$ is a $K$-valuation basis of $V$ over 
$W_s:= W\oplus \Span_K(x_1,\ldots,x_s).$ If $s=$dim$_K W'/W$, then $W_s=W'$.
 Hence $A=\{x_1,\ldots,x_s\}$ and $B'=B_s$ satisfy the assertion of the lemma. 
Otherwise, there is some $x'\in W'\setminus W_s$. Since $x'\in V\setminus W_s$, by Lemma~\ref{val_basis_dim1} we have that there
 is some $a_{s+1}\in W_s$ and $b_{s+1}\in B_s$ such that for $x_{s+1}=x'-a_{s+1}$ the set $\{x_{s+1}\}$ is $K$-valuation independent over $W_s$ and 
$B_{s+1}:=B_s\setminus\{b_{s+1}\}=B\setminus \{ b_1,\ldots,b_{s+1}\}$ form a $K$-valuation basis of $V$ over 
\[
W_{s+1}:=W_s\oplus \Span_K(x')=W_s\oplus \Span_K(x_{s+1})= W\oplus \Span_K(x_1,\ldots,x_{s+1}). 
\]
By Lemma~\ref{lem:basis_transitivity} the elements $x_1,\ldots,x_{s+1}$ form a $K$-valuation basis of $W_{s+1}$ over $W$. 
 
Since dim$_K W'/W$ is finite, the above construction finishes after finitely many steps. 
\end{proof}

A similar construction yields the following.  

\begin{corollary}[$\dagger$]  
Assume that $V$ admits a $K$-valuation basis over $W$. Then every subspace $W'$ of $V$ which contains $W$ and such that $\dim_K W'/W$ is countable admits a $K$-valuation basis over $W$.
\end{corollary}

In contrast, the finiteness assumption on $W'/W$ is necessary to ensure the existence of a $K$-valuation basis of $(V,v)$ over $W'$ as is shown by the following example (cf. \cite[Example 3.62]{K}). Let $K$ be a trivially valued field and $L=K(\!( t )\!)$ with the $t$-adic valuation. Let $V:=\Span_K(t^i\mid i\in \mathbb{N})$ and $V':=\Span_K(t^i-t^{i+1} \mid i\in \mathbb{N})$. It is easy to check that $t\in V\setminus V'$ and furthermore that $v(t-V')$ has no maximal element. By Corollary \ref{cor:Kvalbasis}, $V$ does not admit a $K$-valuation basis over $V'$. Nonetheless, the set $\{t^i\mid i\in\mathbb{N}\}$ is a $K$-valuation basis of $V$.  

\begin{lemma}[$\dagger$]\label{lem:1-ext}Assume that for $x\in V\setminus W$ the set $v(x-W)$ admits a maximum. Then the space $\Span_K(x)$ admits a $K$-valuation basis over $W$. In particular, if for every $x\in V\setminus W$ the set $v(x-W)$ admits a maximum, then every $K$-subspace $W'$ of $V$ of dimension 1 over $W$ admits a valuation basis over $W$.
\end{lemma}

\begin{proof}
Take an element $a\in W$ such that $v(x-a)=\max v(x-W)$ and set $b:=x-a$. We show that $\{b\}$ is $K$-valuation independent over $W$, that is,  
\begin{equation}\label{val_min} \stepcounter{eqn}\tag{E\arabic{eqn}}	
v(cb+w)=\min\{v(cb),v(w)\}
\end{equation}
for every $w\in W$ and $c\in K$. Dividing if necessary by $c$, it is enough to prove equality~\eqref{val_min} for $c=1$. If $v(b)\neq v(w)$, then obviously equality~\eqref{val_min} holds. If~$v(w)=v(b)$, then 
 \[
v(b)=\min\{v(w),v(b)\}\leqslant v(b+w)=v(x-(a-w))\leqslant \max v(x-W)=v(b),
 \]
 as $a-w\in W$. Thus again equality~\eqref{val_min} holds.
\end{proof}

\begin{corollary}[$\dagger$]\label{lem:1-ext-variant}
Let $B=\{b_1,\ldots,b_n\}\subseteq V$ be a $K$-valuation independent set. Let $W:=\Span_K(B)$, $b\in V\setminus W$ and $a\in W$ be such that $v(b-a)=\max v(b-W)$. Then, for $b_{n+1}:=b-a$, $B\cup\{b_{n+1}\}$ is $K$-valuation independent. 
\end{corollary}

\begin{proof}
By Lemma \ref{lem:1-ext}, $\{b_{n+1}\}$ is $K$-valuation independent over $W$. To conclude, using Lemma~\ref{lem:basis_transitivity}, we have that the set $B\cup\{b_{n+1}\}$ is $K$-valuation independent.
\end{proof}





\begin{corollary}[$\dagger$] There is a $K$-subvector space $W'$ such that $W\subseteq W'\subseteq V$, $W'$ admits a $K$-valuation basis over $W$ and $W'\subseteq V$ is an immediate extension. Every maximal subset of $V$ with respect to the property of being $K$-valuation independent over $W$ generates an immediate extension $W'\subseteq V$.
\end{corollary}
\begin{proof}
Let $\mathcal{B}$ denote the collection of all subsets of $V$ which are $K$-valuation independent over $W$. Note that $\mathcal{B}$ is non-empty as the empty set is $K$-valuation independent over $W$. By Zorn's lemma, let $B\in \mathcal{B}$ be a maximal subset. We show the result for $W':=W\oplus\Span_K(B)$. It remains to show that $W'\subseteq V$ is immediate. Suppose for a contradiction it is not. By Lemma \ref{lem:vvs_immediate}, there is $b\in V$ such that $v(b-W')$ has a maximal element. Then by Lemma \ref{lem:1-ext}, there is $c\in V$ such $\{c\}$ is $K$-valuation independent over $W'$. Finally by Lemma \ref{lem:basis_transitivity} this implies that $B\cup\{c\}$ is $K$-valuation independent over $W$ which contradicts the maximality of $B$. The second part of the Corollary is clear from the proof. 
\end{proof}
 

\subsection{Normalized valuation bases}\label{subsec:normal}

Take a subset $B$ of $V$ and consider the following conditions:
\begin{enumerate}[(N1)]
\item for every $b,b'\in B$, if $v(b)$ and $v(b')$ lie in the same coset modulo $vK$, then $v(b)=v(b')$;
\item for every $b\in B$ the system $\Res(B,b)$ is $Kv$-linearly independent;
\item if $b\in B$ and $v(b)\in vK$, then $v(b)=0$;
\item if $b\in B$ and $\res(b)\in Kv$, then $\res(b)=1$.
\end{enumerate}

\begin{lemma} \label{almost_normalized}
Assume that a subset $B$ of $V$ satisfies conditions \textup{(N1)} and \textup{(N2)}. Then $B$ is a $K$-valuation independent set.
\end{lemma}
\begin{proof}
Take $c_1,\ldots,c_n\in K$, $b_1,\ldots,b_n\in B$ and let $I:=\{1,\ldots,n\}$. Let $i_1\in I$ be such that $v(c_{i_1}b_{i_1})=\min\{v(c_ib_i) \,|\, i\in I\}$. Define $J:=\{i\in I\mid v(c_ib_i)=v(c_{i_1}b_{i_1})\}$. We have that 
\begin{equation} \label{eq_vsum} \stepcounter{eqn}\tag{E\arabic{eqn}} 
v\left(\sum_{i\in I} c_ib_i\right)=v(c_{i_1}b_{i_1})+ v\left(\sum_{i\in I} \frac{c_ib_i}{c_{i_1}b_{i_1}}\right).
\end{equation} 
By the choice of $i_1$, we have that $v\left(\frac{c_ib_i}{c_{i_1}b_{i_1}}\right)\geq 0$. For $i\in J$, we have $v(b_i)=v(b_{i_1})$. Indeed, if $v(b_i)\neq v(b_{i_1})$, then (N1) yields that $v(c_ib_i)=v(c_i)+v(b_i)\neq v(c_{i_1})+v(b_{i_1})=v(c_{i_1}b_{i_1})$, contradicting that $i\in J$. Note that this also shows that $v(c_i)=v(c_{i_1})$ for every $i\in J$. If $i\in I\setminus J$, then $v(c_ib_i)\neq v(c_{i_1}b_{i_1})$. Therefore, for every $i\in I\setminus J$, we have that $v\left(\frac{c_ib_i}{c_{i_1}b_{i_1}}\right)> 0$, and consequently $\res\left(\frac{c_ib_i}{c_{i_1}b_{i_1}}\right)=0$. This yields the following identity:
\[
\res\left(\sum_{i\in I} \frac{c_{i}b_{i}}{c_{i_1}b_{i_1}}\right)=\res\left(\sum_{i\in J} \frac{c_{i}b_{i}}{c_{i_1}b_{i_1}}\right)
=\sum_{i\in J} \res\left(\frac{c_{i}}{c_{i_1}}\right)\,\res\left(\frac{b_{i}}{b_{i_1}}\right).
\]
Since $i_1\in J$, one of the coefficients of the linear combination on the right hand side is equal to 1. Therefore, by condition (N2), the whole combination is nonzero. Hence,
\[
 v\left(\sum_{i\in I}^n\frac{c_ib_i}{c_{i_1}b_{i_1}}\right)=0,
\]
which, together with equation~\eqref{eq_vsum}, completes the proof. 
\end{proof}

\begin{definition} \label{normalized_def} 
In view of Lemma \ref{almost_normalized}, a set $B\subset V$ which satisfies conditions (N1)-(N4) will be called a \emph{normalized $K$-valuation independent set}. If in addition $B$ is a basis of $V$, then $B$ is called a \emph{normalized $K$-valuation basis of $V$}.
\end{definition}

\begin{lemma} \label{approxvi} Let $u_1,\ldots,u_n$ and $w_1,\ldots,w_n$ be elements of $L$. If
$\{u_1,\ldots,u_n\}$ is a normalized $K$-valuation independent set and $v(u_i-w_i)>v(u_i)$ for $1\leqslant i\leqslant n$, then $\{w_1,\ldots,w_n\}$ is
a normalized $K$-valuation independent set.
\end{lemma}

\begin{proof}
Straightforward. 
\end{proof}

The next lemma shows that condition (N2) can be replaced by the assumption that $B$ is $K$-valuation independent.
\begin{lemma}\label{lem:val_implies_N2}
Assume that $B\subset V$ is a $K$-valuation independent set. Then $B$ satisfies condition \textup{(N2)}. 
\end{lemma}
\begin{proof}

Take $b, b_1,\ldots,b_n\in B$ such that $v(b_i)=v(b)$ for all $1\leqslant i\leqslant n$. Let $I$ denote the set $\{1,\ldots,n\}$. 
Suppose that 
\[
\sum_{i\in I} \res(c_i) \,\res\left(\frac{b_i}{b}\right)=0
\]
for $c_i\in\OO_K$. Then
\begin{equation*} 
v\left( \sum_{i\in I} c_i\frac{b_i}{b}\right)>0. 
\end{equation*}	
Since the set $B$ is $K$-valuation independent, there is $i_1\in I$ such that 
\[
v\left( \sum_{i\in I} c_ib_i\right)=v(c_{i_1}b_{i_1})=\min_{i\in I}\{v(c_ib_i)\}.
\]
Therefore, since $v(b_i)=v(b)$ for all $i\in I$, we have that 
\[
0<v\left( \sum_{i\in I} c_i\frac{b_i}{b}\right)=v\left( \sum_{i\in I} c_ib_i\right) -v(b) =v(c_{i_1}b_{i_1})-v(b)\leqslant v(c_{i}b_{i})-v(b)=v(c_i).  
\]
Thus $v(c_i)>0$ and consequently $\res(c_i)=0$ for every $i\in I$.
\end{proof}

\begin{lemma}\label{lem:existence-normalized} Let $B=\{b_i \,|\, i\in I\}\subset V$ be a $K$-valuation independent set. Then there is a set $\{c_i \,|\, i\in I\}\subset K^{\times}$ such that $B':=\{c_ib_i \,|\, i\in I\}$ is a normalized $K$-valuation independent set. 
\end{lemma}
\begin{proof}
Take a subset $J$ of $I$ such that the elements $b_j\in B,\, j\in J$ are representatives of the cosets $v(b)+vK, \, b\in K$. Then for every $j\in J$ such that $v(b_j)\notin vK$ set $c_j=1$. If $i\in I\setminus J$ is such that $v(b_i)\notin vK$, take any $c_i\in K$ such that $v(c_i)+v(b_i)=v(b_j)$ for some $j \in J$. It remains to consider the elements $b_i\in B$ with $v(b_i)\in vK$. For every such $b_i$, take $x_i\in K$ such that $v(b_i)=-v(x_i)$. If $\res(x_ib_i)\notin Kv$, set $c_i=x_i$. Otherwise, choose $a_i\in K$ such that $\res(x_ib_i)=\res(a_i)$ and set $c_i=\frac{x_i}{a_i}$. By construction, $B'=\{c_ib_i \,|\, i\in I\}$ satisfies conditions (N1), (N3) and (N4). Since $c_i\neq 0$ for all $i\in I$, by part (3) of Lemma \ref{val_basis_propert}, the set $B'$ is $K$-valuation independent. Therefore, by Lemma \ref{lem:val_implies_N2}, condition (N2) holds. 
\end{proof}

Note that every standard $K$-valuation independent set is also a normalized valuation independent set. The next lemma 
says more about the relations between these two notions.

\begin{lemma}\label{valind_relations} Let $B=\{b_i\in L \,|\, i\in I\}$ be a normalized $K$-valuation basis of $L$. Then  
\begin{enumerate}
\item there are subsets $X,Y\subseteq B$ such that:  
\begin{enumerate}
\item[(i)] $v(X)$ forms a system of representatives of the cosets of $vL$ modulo $vK$ and there is a bijection between $X$ and $v(X)$;
\item[(ii)] the image of $Y$ under the residue map forms a basis of the extension $Lv|Kv$ and there is a bijection between $Y$ and $\res(Y)$,  
\end{enumerate}
\item if $L|K$ is a finite extension and $X,Y$ are as in part $(1)$, then the set $B':=\{xy\mid x\in X, y\in Y\}$ is a $K$-valuation basis of $L$. In particular, if $1\in X$, $1\in Y$, then $B'$ is a standard $K$-valuation basis of $(L|K,v)$.
\end{enumerate}
\end{lemma}
\begin{proof}
From part $(1)$ of Lemma~\ref{val_basis_propert} we infer that $v(L)= v(K)+v(B)$. This shows the existence of the set $X$. Part $(2)$ of Lemma~\ref{val_basis_propert} implies that 
\[
Lv= \res(L,1)=\Span_{Kv}(\res(B,1)),
\]
which proves the existence of the set $Y$. This shows part $(1)$. 

Assume now that $L|K$ is a finite extension. Set $e=(vL:vK)$ and $f=[Lv:Kv]$. By part $(1)$, we have that $\vert X\vert = e$ and $\vert Y\vert=f$. By Lemma \ref{algvind_1}, the set $B'$ is $K$-valuation independent. Let us show that $\vert B'\vert= ef$. Clearly $\vert B'\vert\leqslant ef$. Now for $i=1,2$, let $x_i\in X$ and $y_i\in Y$ be such that $x_1y_1=x_2y_2$. This implies that $v(x_1)=v(x_1y_2)=v(x_2y_2)=v(x_2)$, and since we have a bijection between $X$ and $v(X)$, we must have that $x_1=x_2$. Since $L$ is a field, this implies that $y_1=y_2$, showing that $\vert B'\vert=\vert X\vert\vert Y\vert=ef$. 

It remains to show that $V:=\Span_K(B')=L$. Take any element $b$ of $B$. Then there is $x\in X$ such that $v(b)\in v(x)+vK$. By condition (N1) this means that $v(b)=v(x)$. Hence, $B$ is the disjoint union of the sets $B_x=\{b\in B\mid v(b)=v(x)\}$ where $x$ ranges in $X$. Fix $x\in X$ and let $b_1,b_2\in B_x$. Condition (N2) implies that $\res(\frac{b_1}{x})\neq \res(\frac{b_2}{x})$. Thus 
\[
\vert B_x\vert =\left\vert\left\{\res\left(\frac{b}{x}\right)\mid v(b)\in B_x\right\}\right\vert= \vert \res(B,x)\vert \leqslant \vert Y\vert=e. 
\]
This shows that
\[
[L:K]=|B|= \sum_{x\in X}|B_x| \leqslant e\cdot f = |B'|=\dim_KV,
\]
and consequently that $V=L$. The last assertion follows immediately. This shows part (2). 
\end{proof}

Note that if $\{b_i\, |\, i\in I\}$ is a $K$-valuation basis of $L$, then for any nonzero element $b$ of $L$ the set 
$\{\frac{b_i}{b}\, |\, i\in I\}$ is also a $K$-valuation basis of $L$. Thus if $L$ admits a $K$-valuation basis, then it admits also a $K$-valuation basis containing the unity. 

Assume that $L|K$ is a finite extension and $L$ admits a $K$-valuation basis $B$. Lemma~\ref{lem:existence-normalized} yields that $L$ admits a normalized $K$-valuation basis $B'$. Moreover, since we can assume that $B$ contains $1$, without loss of generality we can assume that $B'$ also contains $1$ (cf. the proof of part (1) of Lemma~\ref{valind_relations}). Hence, by part $(2)$ of Lemma~\ref{valind_relations} we obtain that $L$ admits a standard $K$-valuation basis. In view of Lemma \ref{algvind_1} we just proved the following corollary:

\begin{corollary} \label{valbasis_stand}
Assume that the extension $(L|K,v)$ is finite. Then $L$ admits a $K$-valuation basis if and only if it admits a standard $K$-valuation basis.
\end{corollary}

We finish this section with the following characterization of immediate extensions. 

\begin{proposition}\label{pro:char-immediate}
Take a valued field extension $(M|K,v)$. Then the following are equivalent:
\begin{enumerate}
\item $(M,v)$ is an immediate extension, 
\item for every subset $\{b_1,\ldots,b_n\}$ of some valued field extension of $(M,v)$, if $b_1,\ldots,b_n$ are $K$-valuation independent, then they are also $M$-valuation independent. 
\end{enumerate}
\end{proposition}

\begin{proof}
Suppose $(M|K,v)$ is an immediate extension and let $B=\{b_1,\ldots,b_n\}$ be a $K$-valuation independent set. By Lemma \ref{valind_relations}, there are $c_1,\ldots,c_n\in K^\times$ such that $B'=\{c_1b_1,\ldots,c_nb_n\}$ is a normalized $K$-valuation independent set. It is enough to show that $B'$ is $M$-valuation independent. Indeed, if it is, then for $a_1,\ldots,a_n\in M$ we have that 
\[
v\left( \sum_{i=1}^n a_ib_i \right)= v\left( \sum_{i=1}^n (a_ic_i^{-1}) c_ib_i \right)=\min_{1\leqslant i \leqslant n} \{v((a_ic_i^{-1}) c_ib_i)\}=\min_{1\leqslant i \leqslant n}\{v(a_ib_i)\}. 
\] 
Since $(M|K,v)$ is immediate, we have that $vK=vM$ and that $Kv=Mv$, hence $B'$ also satisfies all properties (N1)-(N4) with respect to the valued field $(M,v)$. Therefore, by Lemma \ref{almost_normalized}, $B'$ is $M$-valuation independent. 

For the converse, suppose $(M|K,v)$ is not an immediate extension and let $a\in M\setminus K$ be such that $\max v(a-K) = v(a-b)$ for some $b\in K$ (by Lemma \ref{lem:vvs_immediate} or Theorem \ref{charact_immediate}). Since $\{1\}$ is a $K$-valuation independent set, by Corollary \ref{lem:1-ext-variant}, $\{a-b,1\}$ is also $K$-valuation independent. Nevertheless, $\{a-b,1\}$ is not $M$-valuation independent as 
\[
v((a-b)^{-1}(a-b)+(-1)(1))=v(0)>\min(v(a-b),v(1)), 
\] 
which completes the proof.
\end{proof}


\section{Defectless and $vs$-defectless extensions}\label{sec:defectext}

Let $(L|K,v)$ be a valued field extension and suppose that the field extension is finite. Lemma \ref{algvind_1} shows in particular that 
\begin{equation} \label{fund_inequality} \stepcounter{eqn}\tag{E\arabic{eqn}} 
[L:K]\geq (vL:vK)[Lv:Kv].
\end{equation} 
In fact, if $v=v_1,\ldots, v_m$ are the distinct extensions of the valuation $v$ on $K$ to the field $L$, the so-called \emph{Lemma of Ostrowski} (see~\cite[Chapter VI, \S 12, Corollary to Theorem~25]{zariski-samuel}) establishes that 
\begin{equation*}
\label{ostrowski}
[L:K]=\sum_{i=1}^m p^{n_i}(v_iL:v_iK)[Lv_i:Kv_i], 
\end{equation*}		 
where $p$ denotes the \emph{characteristic exponent} of $Kv$ (that is,  $p=$ char$ Kv$ if it is positive and $p=1$ otherwise) and for each $i\in\{1,\ldots,m\}$, $n_i$ is a non-negative integer. The factor $p^{n_i}$ is called the \emph{defect} of the valued field extension $(L|K,v_i)$. If $p^{n_i}=1$ for all $i\in\{1,\ldots,m\}$, we say that \emph{$L$ is a defectless field extension of $(K,v)$}. Otherwise we speak of a \emph{defect extension}. We will center our study of defectless extensions to the particular case where $m=1$, that is, where the valuation $v$ extends uniquely to $L$. Note that every subextension of a finite defectless extension is again defectless. 

An infinite algebraic extension $(L|K,v)$ such that the valuation $v$ admits a unique extension from $K$ to $L$ is called \emph{defectless} if every finite subextension $(F|K,v)$ of $(L|K,v)$ is defectless. Observe that if $(F|K,v)$ is a subextension of an infinite defectless extension $(L|K,v)$, then $(F|K,v)$ is defectless. However, the extension $(L|F,v)$ may not be defectless. 

A valued field $(K,v)$ is called \emph{defectless} if every finite extension of $K$ is defectless. In particular, any valued field $(K,v)$ of residue characteristic zero is defectless. If $(K,v)$ is a henselian defectless field, then it is called \emph{algebraically complete}.

The following proposition shows the relation between defectless and $vs$-defectless extensions (see Definition \ref{def:vsdefect}).

\begin{proposition} \label{char_defectless}
Assume that the extension $(L|K,v)$ is finite. Then the following conditions are equivalent.
\begin{enumerate}
\item $[L:K]=(vL:vK)[Lv:Kv]$,
\item $(L|K,v)$ admits a standard $K$-valuation basis,
\item $(L|K,v)$ admits a $K$-valuation basis, 
\item $(L|K,v)$ is a $vs$-defectless extension.
\end{enumerate}
\end{proposition}
\begin{proof}
$(1)\to (2):$ Set $e=(vL:vK)$, $f=[Lv:Kv]$ and assume that $[L:K]=ef$. Let $1=x_1,\ldots,x_e$ be representatives of the cosets of $vL$ modulo $vK$ and $1=y_1,\ldots,y_f\in \OO_L$  be such that their residues form a basis of the extension $Lv|Kv$. By Lemma \ref{algvind_1}, $\{x_iy_j\, |\, 1\leqslant i\leqslant e, 1\leqslant j\leqslant f\}$ forms a standard $K$-valuation independent set. Since the set contains $ef=[L:K]$ elements, it is a basis of $L|K$, which proves $(2)$.

$(2)\to (3):$  Trivial.

$(3)\to (4):$  This follows by Lemma \ref{lem:val_bas_subext}. 

$(4)\to (1):$  Since the extension is finite, $L$ admits a $K$-valuation basis. Then by Corollary~\ref{valbasis_stand}, it admits a standard $K$-valuation basis $B$. From the definition of the standard $K$-valuation independent set it follows that \mbox{$|B|=(vL:vK)[Lv:Kv]$.} As $B$ is a basis of $L$ as a $K$-vector space, $(1)$ holds. 
\end{proof}
Note that condition $(1)$ of the previous Proposition states that the valuation $v$ extends in a unique way from $K$ to $L$ and  that $(L|K,v)$ is defectless. 

\begin{corollary}\label{cor:vs-def-def}
Assume that the extension $L|K$ is algebraic. Then $(L|K,v)$ is $vs$-defectless if and only if $v$ extends in a unique way from $K$ to $L$ and $(L|K,v)$ is defectless. 
\end{corollary}
\begin{proof}
The assertion follows directly from Proposition~\ref{char_defectless} together with the definition of defectless extension in the case of infinite algebraic extensions of valued fields. 
\end{proof}

Note that the above corollary yields the following characterization of algebraically complete fields.
\begin{corollary}
A valued field $(K,v)$ is algebraically complete if and only if every finite extension of $K$ admits a $K$-valuation basis. 
\end{corollary}

Let us now present and prove the main theorem of this section. 

\begin{theorem}\label{thm:main} Consider the following properties of a valued field extension $(L|K,v)$:
\begin{enumerate}[(A)]
\item the extension $(L|K,v)$ is $vs$-defectless;
\item for every finitely generated $K$-vector space $W\subseteq L$ and any element $b\in L$, $v(b-W)$ has a maximal element. 
\item $L$ is linearly disjoint over $K$ to every immediate extension $(M,v)$ of $(K,v)$, both $L$ and $M$ being contained in a common valued field extension. 
\end{enumerate}
Then $(A)\Leftrightarrow (B) \Rightarrow (C)$. 
\end{theorem}

\begin{proof}
$(A)\Rightarrow(B)$ Let $W$ be a finitely generated $K$-vector subspace of $L$ and $b\in L$. If $b\in W$, then $\infty$ is the maximal element of $v(b-W)$. So suppose that $b\notin W$ and set $V:=W\oplus Kb$. By assumption, $V$ has a $K$-valuation basis. By Lemma \ref{lem:val_bas_subext}, $V$ has a valuation basis over $W$. Therefore, $v(b-W)$ has a maximal element by Lemma \ref{lem:vbasis_max}. 

 $(B)\Rightarrow(A)$ Let $V\subseteq L$ be a $K$-vector space of dimension $n\geq 1$. We prove that $V$ has a $K$-valuation basis by induction on $n$. For $n=1$ there is nothing to prove as any non-zero vector of $V$ is a $K$-valuation basis. Suppose that $\dim(V)=n+1$ and let $W\subseteq V$ be a subspace of dimension $n$. Let $B=\{b_1,\ldots,b_n\}$ be a valuation basis of $W$ and $a\in V\setminus W$. By assumption, there exists $w\in W$ such that $v(a-w)=\max v(a-W)$. Then, by Corollary \ref{lem:1-ext-variant}, the set $\{b_1,\ldots,b_{n+1}\}$ where $b_{n+1}=a-w$, is $K$-valuation independent, which completes the proof. 

$(A)\Rightarrow (C)$ Let $B\subseteq L$ be a finite $K$-linearly independent set. To show that $B$ is $M$-linearly independent it is enough to show that there is an $M$-linearly independent basis of $W:=\Span_K(B)$. By assumption, $W$ has a $K$-valuation basis $B'$. Then by Proposition \ref{pro:char-immediate} $B'$ is also $M$-valuation independent, hence $M$-linearly independent by Remark \ref{rmk:valInd}. 
\end{proof}

We finish this section by showing that the implication $(C)\Rightarrow (A)$ does not hold in general. Recall that a valued field $(K,v)$ is said to be \emph{algebraically maximal} if it does not admit any proper immediate algebraic extension. 

\begin{proposition}\label{prop:notCA}
Take a valued field $(K,v)$ which is not algebraically maximal. Then there exists a simple transcendental extension 
$(K(x)|K,v)$ which is linearly disjoint from each maximal immediate extension of $(K,v)$ in every common valued 
field extension although the $K$-vector space generated by $1$ and $x$ does not admit a $K$-valuation basis.
\end{proposition}
\begin{proof}
By our assumption there exists a nontrivial immediate algebraic extension $(K(a)|K,v)$. Take an extension 
$(K(a,y)|K(a),v)$ such that $v(y)>vK$ and set $x:=a+y$. Then $v(x-K)=v(a-K)$ has no maximal element (by Theorem \ref{charact_immediate}) and therefore, by Theorem \ref{thm:main} the extension $(K(x)|K,v)$ is not $vs$-defectless (in fact, by Corollary \ref{cor:Kvalbasis} and Lemma \ref{lem:val_bas_subext}, the $K$-vector space generated by $1$ and $x$ does not admit a $K$-valuation basis). On the other hand, $K(x)$ is linearly disjoint from each maximal immediate extension of $(K,v)$ since otherwise, by \cite[VIII, \S 3, Proposition~3.3]{lang2002}, $x$ would be algebraic over some such extension $(M,v)$. But this is impossible because if $f$ is the minimal polynomial of $a$ over $K$, then $vf(x)>vK=vM$, so $vM(x)$ would contain an element that is not torsion over $vM$.
\end{proof}


\section{Instances of $(C)\Rightarrow (A)$}\label{sec:CA}

The aim of this section is to show that various natural classes of valued fields do satisfy the implication $(C)\Rightarrow (A)$. We will need the following theorem from \cite{B}:

\begin{theorem}[Baur]\label{thm:baur} Let $(K,v)$ be maximal. Then every extension $(L|K,v)$ is $vs$-defectless. In particular, every algebraic extension is defectless. 
\end{theorem}

\begin{remark}
Baur's theorem is an enhanced version of the following theorem:\sn
{\it If $(K,v)$ is a maximal field and $(L|K,v)$ is a finite extension, then also $(L,v)$ is a maximal field, and the 
extension $(L|K,v)$ admits a valuation basis.}
\sn
The history of this theorem is not entirely clear. In the past, we have worked with the following two references: Warner's book on Topological Fields (\cite[Theorems 31.21 and 31.22]{W}), and Ribenboim's 1968 monograph on Valuation Theory (\cite[Th\'eor\`eme 1, p.~ 230]{Ri}). Ribenboim credits Krull, and the same is done in \cite{ADH}. But although Krull was apparently the first to prove the existence of maximal immediate extensions in \cite{Kr}, we did not find the above theorem in that paper. In contrast, Warner credits Kaplansky for proving the theorem in his thesis, but apparently this part of the thesis was never published.

Ribenboim in his proof uses the fact that a valued field is maximal if and only if every pseudo Cauchy sequence has a limit. The proof is relatively straightforward, but very technical. Warner uses the notion of ``linearly compact module''. In more modern terms, this translates to the notion of spherical completeness, and one can use that valued fields are maximal if and only if their underlying ultrametric spaces are spherically complete. It turns out that the gist of the two proofs actually is the fact that finite products of spherically complete ultrametric spaces are again spherically complete, as proven in \cite[Proposition 10]{Ku}. In that paper, this is used to prove that the multidimensional Hensel's Lemma holds for every maximal valued field, which then by a quick argument implies that it also holds in every henselian field.
In their recent book \cite{ADH}, in Corollary 3.2.26, Aschenbrenner, van den Dries and van der Hoeven use the theorem on products of
spherically complete ultrametric spaces to give a short and elegant proof of Baur's theorem, and thereby a nice alternative to the proofs of the above cited theorem given by Ribenboim and Warner.
\end{remark}

\begin{corollary}
A valued field is maximal if and only if every extension of the field is $vs$-defectless.
\end{corollary}
\begin{proof}
If $(K,v)$ is a maximal field, then by the above theorem it is also $vs$-defectless. For the converse, suppose that $(K,v)$ is not maximal, so it admits a nontrivial immediate extension $(L|K,v)$. Since by Lemma~\ref{lem:immediate_valdep} every two elements of $L$ are $K$-valuation dependent, the extension is not $vs$-defectless. 
\end{proof}

\subsection{Abstract criterion}

We are ready to prove the sufficiency of the abstract criterion mentioned in the introduction. Standard results in model theory will be involved in the proof and we refer the reader to \cite{TZ} for the needed background. 

\begin{theorem}\label{thm:criterion}
Consider an elementary class $\cal K$ of valued fields having the following properties:
\begin{enumerate}
\item[(P1)] every member of $\cal K$ is existentially closed in each of its maximal immediate extensions,
\item[(P2)] all maximal immediate extensions of members of $\cal K$ are again members of $\cal K$,
\item[(P3)] if $(K,v)\in \cal K$ and $(F,v)$ is a relatively algebraically closed subfield such that $(K|F,v)$ is
immediate, then $(F,v)\in \cal K$.
\end{enumerate}
Take $(K,v)\in \cal K$. Then every valued field extension $(L|K,v)$ satisfies $(C)\Rightarrow (A)$. 
\end{theorem}

\begin{proof}
Take a highly saturated elementary extension $(L^*|K^*,v)$ of $(L|K,v)$. Then also $(K^*,v)$ is a highly saturated
elementary extension of $(K,v)$. Since $(K,v)$ is existentially closed in each maximal immediate extension, every
such extension embeds in $(K^*,v)$ over $K$. So we may assume that there is a maximal immediate extension $(M,v)$
of $(K,v)$ inside of $(K^*,v)$. We note that $(M,v)\in \cal K$ by property (P2) of $\cal K$. We wish to show that $(L.M|L,v)$ is an immediate extension.

By Zorn's Lemma, there exists a subextension $(M_0,v)$ of $(K,v)$ in $(M,v)$ maximal with the property that
$(L.M_0|L,v)$
is an immediate extension. We show first that $M_0$ is relatively algebraically closed in $M$. Take $M_1$ to be the
relative algebraic closure of $M_0$ in $M$. Then $(L.M_1|L.M_0,v)$ is an algebraic extension within $(L^*,v)$, so
$v(L.M_1)/v(L.M_0)$ is a torsion group and $(L.M_1)v|(L.M_0)v$ is an algebraic extension. Since
$(L.M_0|L,v)$ is immediate, it follows that $v(L.M_1)/vL$ is a torsion group and $(L.M_1)v|Lv$ is an algebraic
extension. But as $(L^*,v)$ is an elementary extension of $(L,v)$, also $vL^*$ is an elementary extension of $vL$
and $L^*v$ is an
elementary extension of $Lv$. It follows that $vL^*/vL$ is torsionfree and $L^*v|Lv$ is regular, so $v(L.M_1)=vL$
and $(L.M_1)v=Lv$, showing that $(L.M_1|L,v)$ is an immediate extension. Hence $M_1=M_0$ by the maximality of
$M_0\,$, i.e., $M_0$ is relatively algebraically closed in $M$. Since $(M|M_0,v)$ is immediate like $(M|K,v)$, we
obtain from property (P3) of the class $\cal K$ that $(M_0,v)\in\cal K$.

Suppose that there is some $x\in M\setminus M_0\,$. Then by \cite[Theorem~1]{kaplansky}, $x$ is the limit of a pseudo
Cauchy sequence in $(M_0,v)$ that does not have a limit in $M_0\,$. If this sequence would be of algebraic type,
then by \cite[Theorem~3]{kaplansky} there would exist a nontrivial immediate algebraic extension of $(M_0,v)$.
Passing to some maximal immediate
extension thereof, we would obtain a maximal immediate extension of $(M_0,v)$ in which $M_0$ is not relatively
algebraically closed. But this contradicts property (P1). We
conclude that the pseudo Cauchy sequence is of transcendental type.

Suppose that the pseudo Cauchy sequence has a limit $y$ in $(L.M_0,v)$. Then from \cite[Theorem~2]{kaplansky} it
follows that $(M_0(y)|M_0,v)$ is immediate. Take any maximal immediate extension $(M',v)$ of $(M_0(y),v)$; since
$(M_0(y)|M_0,v)$ and $(M_0|K,v)$ are immediate, $(M',v)$ is also a maximal immediate extension of $(K,v)$. But
$M_0(y)$ is not linearly disjoint from $M'$ over $M_0$ and thus $L$ is not linearly disjoint from $M'$ over $K$ (see \cite[VIII, \S 3, Proposition 3.1]{lang2002}),
which contradicts our assumptions. We have shown that the pseudo Cauchy sequence has no limit in $(L.M_0,v)$.

Again, \cite[Theorem~2]{kaplansky} implies that $(L.M_0(x)|L.M_0,v)$ is immediate. As also $(L.M_0|L,v)$ is immediate,
we find that $(L.M_0(x)|L,v)$ is immediate. But since $x\notin M_0$, the extension $M_0(x)|M_0$ is nontrivial,
which contradicts the maximality of $M_0\,$. We conclude that there is no such $x$, so $M_0=M$ and $(L.M|L,v)$ is
immediate.

\

Now take any $u_1,\ldots,u_n\in L$ that are $K$-linearly independent, and denote the $K$-vector space they generate
by $V$. By our assumptions, they remain $M$-linearly
independent. By Theorem \ref{thm:baur}, the $M$-vector space generated by them admits a valuation basis, and hence by Lemma~\ref{lem:existence-normalized} also a normalized $M$-valuation basis $u'_1,\ldots,u'_n\in \Span_M(u_1,\ldots,u_n)\subseteq L.M$. We write
\[
u'_i\>=\> \sum_{j=1}^n d_{ij} u_j \qquad\qquad (1\leq i\leq n)
\]
with $d_{ij}\in M\subset K^*$. Since $(L.M|L,v)$ is immediate, we can choose elements $w'_1,\ldots,w'_n\in L$ such
that $v(u'_i-w'_i)>vu'_i$ for $1\leq i\leq n$. As $u'_1,\ldots,u'_n$ are in particular normalized $K$-valuation
independent, the same holds for $w'_1,\ldots,w'_n$ by Lemma~\ref{approxvi}.

Let $E$ denote the predicate for the smaller field in the pairs $(L^*|K^*,v)$ and $(L|K,v)$. Consider the
existential sentence stating the existence of elements $x_{ij}$ with $E(x_{ij})$, $1\leq i,j\leq n$, and such that
\begin{equation}                           \label{eqappr} \stepcounter{eqn}\tag{E\arabic{eqn}} 
v \left(\sum_{j=1}^n x_{ij} u_j\,-\,w'_i\right)\> >\> vw'_i\qquad\qquad (1\leq i\leq n)
\end{equation}
with parameters $u_i$ and $w'_i$ in $(L|K,v)$. It holds in $(L^*|K^*,v)$ for $x_{ij}=d_{ij}$, and since
$(L^*|K^*,v)$ is an elementary extension of $(L|K,v)$, it also holds in $(L|K,v)$. That is, there are $c_{ij}\in K$
such that the equations \eqref{eqappr} hold with $x_{ij}=c_{ij}$. We set
\[
w_i\>:=\> \sum_{j=1}^n c_{ij} u_j \>\in\> V \qquad\qquad (1\leq i\leq n).
\]
We have that $v(w_i-w'_i)>v(w'_i)$ for $1\leq i\leq n$, hence it follows from Lemma~\ref{approxvi} that
$w_1,\ldots,w_n$ are normalized $K$-valuation independent and hence form a normalized $K$-valuation basis for $V$.
We have now proved that the extension $(L|K,v)$ is vs-defectless.
\end{proof}

Classes $\mathcal{K}$ that satisfy the hypothesis of the previous theorem include:

\begin{enumerate}
\item[$\bullet$] The class of all tame valued fields. The fact that they form an elementary class and properties (P1)-(P3) follow from results by the third author in \cite{kuhlmann2016} (see in particular Lemma 3.7 and Theorems 1.4 and 3.2). By \cite[Corollary 3.3]{kuhlmann2016}, the class of tame fields includes all algebraically maximal Kaplansky fields (see later Definition \ref{def:kaplanskyfield}), hence in particular all henselian valued fields of residue characteristic 0. Note moreover that all algebraically closed valued fields are examples of algebraically maximal Kaplansky fields. 
\item[$\bullet$] The class of all henselian finitely ramified fields (which includes the class of all $\wp$-adically closed fields). They key property to show is (P1) which can be found in \cite[Corollary. 8.23]{fvkthesis} (see also \cite[Theorem 8.9]{fvkthesis} which builds on work of Er\v{s}ov \cite{ersov} and Ziegler \cite{zig}).  
\end{enumerate}

\subsection{Two further cases}\label{subsec:kaplanskyfields}

We finish with two further cases where implication $(C)\Rightarrow (A)$ holds. The first one includes the case of discretely valued field extensions of rank 1. 

\begin{theorem}\label{thm:cofinal} Let $(L|K,v)$ be such that: 
\begin{enumerate}
\item $\widehat{K}$ (the completion of $K$) is the maximal immediate extension of $K$ and 
\item $vK$ is cofinal in $vL$. 
\end{enumerate}
Then $(C)\Rightarrow (A)$. 
\end{theorem}

\begin{proof}
Suppose that $(C)$ holds and let $V\subseteq L$ be a $K$-vector space such that $\dim_K(V)=n$. Let $\{b_1,\ldots,b_n\}$ be a $K$-basis of $V$. By condition $(C)$, we have that $\{b_1,\ldots,b_n\}$ is also $\widehat{K}$-independent. Let $W:=\Span_{\widehat{K}}(b_1,\ldots,b_n)$. Since $\widehat{K}$ is maximal, by Theorem \ref{thm:baur}, $W$ has a $\widehat{K}$-valuation basis $\{b_1',\ldots,b_n'\}$. Let $I$ denote the set $\{1,\ldots, n\}$. For $i\in I$, let $c_{ij}\in \widehat{K}$ be such that $b_i'=\sum_{j\in I} c_{ij}b_j$. 

For each pair $(i,j)\in I^2$, there is a Cauchy sequence $(c_{ij}^\alpha)_{\alpha<\lambda_{ij}}$ in $K$ with limit $c_{ij}$. Since $vK$ is cofinal in $vL$ and the values in $\{v((c_{ij}^{\alpha}-c_{ij})b_j)\mid \alpha<\lambda_{ij}\}$ are either cofinal in $vL$ or contain $\{\infty\}$, for each pair $(i,j)\in I^2$, there is $\alpha_{ij}<\lambda_{ij}$ such that 
\begin{equation}\label{equation1} \stepcounter{eqn}\tag{E\arabic{eqn}} 
\min_{j\in I} \{v((c_{ij}^{\alpha_{ij}}-c_{ij})b_j)\} > v\left(\sum_{j\in I} c_{ij}b_j\right)=v(b_i'). 
\end{equation}
Set $b_i^*:=\sum_{j\in I} c_{ij}^{\alpha_{ij}}b_j$. Note that inequality \eqref{equation1} implies in particular that for each $i\in I$,
$v(b_i^*-b_i')>v(b_i^*)=v(b_i')$. Indeed, we have that 
\begin{equation}\label{equation1.5} \stepcounter{eqn}\tag{E\arabic{eqn}} 
v(b_i^*-b_i') = v\left(\sum_{j\in I} (c_{ij}^{\alpha_{ij}}-c_{ij})b_j\right) \geqslant \min_{j\in I}\{v((c_{ij}^{\alpha_{ij}}-c_{ij})b_j)\} > v\left(\sum_{j\in I} c_{ij}b_j\right)=v(b_i').
\end{equation}
We claim that $b_i^*$ is a $K$-valuation basis of $V$. Let $a_1,\ldots,a_n\in K$. From \eqref{equation1.5} it follows that 
\begin{equation}\label{equation2} \stepcounter{eqn}\tag{E\arabic{eqn}} 
v\left(\sum_{i\in I} a_i(b_i^*-b_i')\right)\geqslant \min_{i\in I}\{v(a_i(b_i^*-b_i'))\} > \min_{i\in I}\{v(a_ib_i')\}=v\left(\sum_{i\in I} a_ib_i'\right), 
\end{equation}
where the last equality holds since $b_i'$ are $\widehat{K}$-valuation independent. We therefore have that  
\begin{align*}
v\left(\sum_{i\in I} a_ib_i^*\right) 	& = v\left(\sum_{i\in I} a_ib_i^*-\sum_{i\in I} a_ib_i'+\sum_{i\in I} a_ib_i'\right) & \\
								& = v\left(\sum_{i\in I} a_i(b_i^*-b_i')+\sum_{i\in I} a_ib_i'\right)& \\
								& =\min\{v(a_1b_1'),\ldots,v(a_nb_n')\} & \text{ by \eqref{equation2}} \\
								& =\min\{v(a_1b_1^*),\ldots,v(a_nb_n^*)\}.
\end{align*} 

\end{proof}

Finally, our second case deals with algebraic extensions of Kaplansky fields. Let us recall their definition. 

\begin{definition}\label{def:kaplanskyfield} A valued field $(K,v)$ of residue characteristic $p\geqslant 0$ is called a \emph{Kaplansky field} if it satisfies:
\begin{enumerate}
\item[(K1)] if $p>0$ then the value group $vK$ is $p$-divisible,
\item[(K2)] the residue field $Kv$ is perfect,
\item[(K3)] the residue field $Kv$ admits no finite separable extension of degree divisible by $p$.
\end{enumerate}
\end{definition}
The original ``hypothesis A'' assumed by Kaplansky consisted of condition (K1) and the following property:
\begin{enumerate}
\item[(K2')]  for every additive polynomial $f(X)\in Kv[X]$  and $c\in Kv$ the polynomial $f(X)+c$ has a root in $Kv$.
\end{enumerate}
Then Whaples clarified the meaning of condition (K2') proving that it holds if and only if $Kv$ admits no finite extensions of degree divisible by $p$  {\cite[Theorem 1]{whaples1957}}. This shows the equivalence of conditions (K1)-(K3) with ``hypothesis A''.  

We will need the following Theorem about algebraically maximal Kaplansky fields. 

%
%
%

\begin{theorem}[{\cite[Theorem 1.1]{fvk}}]  \label{subf_max_imm}
Take an algebraically maximal Kaplansky field $(L,v)$ and a subfield $K$ of $L$. Then $L$ contains a maximal immediate algebraic extension of $(K,v)$. Moreover, if $(L,v)$ is maximal, then it also contains a maximal immediate extension of $(K,v)$. 
\end{theorem} 


\begin{theorem}\label{thm:kaplansky algebraic}
Assume that $(K,v)$ is a Kaplansky field. If $(L,v)$ is an algebraic extension of $(K,v)$, then $(C)\Rightarrow  (A)$.
\end{theorem}
\begin{proof}
Assume that $L|K$ is linearly disjoint from any immediate extension $F|K$. We wish to show that $(L|K,v)$ is $vs$-defectless. 
Note that every finitely generated $K$-vector subspace of $L$ is contained in some finite field subextension $E$ of $L|K$. Hence, by Lemma~\ref{lem:val_bas_subext} it is enough to show that every finite field subextension $E|K$ of $L|K$ is $vs$-defectless. 

Take a finite subextension $(E|K,v)$ of $(L|K,v)$. Since $(E,v)$ is an algebraic extension of a Kaplansky field, it is also a Kaplansky field. Take $(M_E,v)$ to be the maximal immediate extension of $(E,v)$. By Theorem~\ref{subf_max_imm}, the field $M_E$ contains a maximal immediate extension $M_K$ of $K$. Thus it contains also $M_K.E$. 
By our assumptions, $M_K$ and $E$ are linearly disjoint over $K$. 
Moreover, as the field $(M_K,v)$ is maximal, Theorem~\ref{thm:baur} together with Proposition~\ref{char_defectless} yields that $[M_K.E:M_K]=(vM_K.E:vM_K)[(M_K.E)v:M_Kv]$. We thus obtain that
\begin{equation}\label{fund_in_1} \stepcounter{eqn}\tag{E\arabic{eqn}} 
(vE:vK)[Ev:Kv]\leqslant [E:K]=[M_K.E:M_K]=(vM_K.E:vM_K)[(M_K.E)v:M_Kv].
\end{equation} 
Since $(M_E|E,v)$ is immediate, the same holds for $(M_K.E|E,v)$. Thus 
\begin{equation} \label{fund_in_2} \stepcounter{eqn}\tag{E\arabic{eqn}}
(vM_K.E:vM_K)[(M_K.E)v:M_Kv] =  (vE:vK)[Ev:Kv].
\end{equation}Equations~(\ref{fund_in_1}) and~(\ref{fund_in_2}) yield that $[E:K]=(vE:vK)[Ev:Kv]$. By Proposition~\ref{char_defectless} we obtain that the extension $(E|K,v)$ is $vs$-defectless.
\end{proof}

Note that Theorem~\ref{thm:criterion} states that implication $(C)\Rightarrow (A)$ holds in particular for any extension $(L,v)$ of an algebraically maximal Kaplansky field $(K,v)$. The above fact shows that we can omit the assumption of being algebraically maximal in the case of algebraic extensions.


\section*{Appendix}\label{sec:appendix}

Given a totally ordered set $S$, we denote by $S_{\infty}$ the set $S$ together with a new element $\infty$ such that $s<\infty$ for all $s\in S$. 

\begin{definition}\label{def:vvs} Let $(K,v)$ be a valued field and $S$ be a totally ordered set. A valued $(K,v)$-vector space $W$ is a $K$-vector space together with a map $\theta\colon W\to S_{\infty}$ and an action of $vK$ on $S_{\infty}$ satisfying 
\begin{enumerate}
\item $\theta(x)=\infty$ if and only $x=0$ for all $x\in W$
\item $\theta(x+y)\geqslant \min\{\theta(x),\theta(y)\}$ for all $x,y\in W$
\item $\theta(ax)=v(a)(\theta(x))$ for all $a\in K^\times$ and all $x\in W$
\item for all $\gamma\in vK$ and $s_1,s_2\in S$, if $s_1<s_2$ then $\gamma(s_1)<\gamma(s_2)$. 
\item for all $\gamma_1,\gamma_2\in vK$ and $s\in S$, if $\gamma_1<\gamma_2$ then $\gamma_1(s)<\gamma_2(s)$. 
\item $v(a)(\infty)=\infty$ for all $a\in K^\times$. 
\end{enumerate}
\end{definition}  

The definitions of a valued vector space as given in \cite{fuchs1975,KS} correspond to the special case of Definition \ref{def:vvs} where $v$ is the trivial valuation on $K$ and the action of $vK$ on $S$ is also trivial, that is, $\theta(ax)=w(x)$ for all $a\in K^\times$ and all $x\in W$. More general frameworks can also be found in \cite{maalouf}. 

In this article we have worked in the special situation where the valued~$(K,v)$-vector spaces come from a valued field extension. Let~$(L|K,v)$ be any such extension. Any~$K$-vector space~$W\subseteq L$ can be endowed with the structure of a valued~$(K,v)$-vector space by taking $\theta=v$, $S=vW$ and the action of~$vK$ on~$S_{\infty}$ as addition, i.e., $v(a)(v(x))=v(a)+v(x)$ for all~$a\in K^\times$ and~$x\in W$. 

\subsection*{Acknowledgements}
During the realization of this project the second author was supported by ERC grant agreement nr. 637027 (TOSSIBERG). The third author was supported by OPUS research grant 2017/25/B/ST1/01815. 

\bibliographystyle{acm}
\bibliography{biblio_defectless}


\end{document}